\documentclass[12pt]{article}
\usepackage{graphicx}
\usepackage{amsmath}
\usepackage{amsfonts}
\usepackage{amsthm}
\usepackage{verbatim}
\usepackage[pdftex]{hyperref}
\usepackage{color}

\usepackage[T1]{fontenc}
\usepackage[utf8]{inputenc}

\newtheorem{theorem}{Theorem}
\newtheorem*{theorem*}{Theorem}
\newtheorem{lemma}[theorem]{Lemma}
\newtheorem{corollary}[theorem]{Corollary}

\theoremstyle{definition}

\newtheorem*{problem*}{Problem}

\DeclareMathOperator{\mex}{mex}

\begin{document}

\title{PRIMES STEP Plays Games}

\author{Pratik Alladi \and Neel Bhalla \and Tanya Khovanova \and Nathan Sheffield \and
Eddie Song \and William Sun \and Andrew The \and Alan Wang \and
Naor Wiesel \and Kevin Zhang \and Kevin Zhao}
\date{}

\maketitle

\begin{abstract}
A group of students in 7-9 grades are inventing combinatorial impartial games. The games are played on graphs, piles, and grids. We found winning positions, optimal strategies, and other interesting facts about the games.
\end{abstract}


\section{Introduction}

This project is a part of the PRIMES STEP program that allows students in grades seven through nine to try research in math. The group was mentored by Tanya Khovanova.  The group studied different aspects of mathematics including combinatorial game theory. Then the students invented and analyzed their own games. Each game has its one flavor.

The \textit{Chocolate Stones} game in Section~\ref{sec:cs} sounds complicated, but it is a no-strategy game. 

The \textit{Diamond} game in Section~\ref{sec:diamonds} is a game with geometrical shapes. Surprisingly, for any non-trivial starting position the second player wins. We also discuss generalizations of this game to other shapes.

The \textit{Demon Money} game in Section~\ref{sec:dm} looks like it might be a no-strategy game, but is not. However, in the optimal play the number of moves is uniquely defined by the starting position.

The \textit{Sum-from-Product} game in Section~\ref{sec:sfp} combines addition and multiplication. The game is difficult to analyze, and we provide the computational answer: the list of P-positions.

The \textit{Remove-a-Square} game in Section~\ref{sec:rs} is another geometric game played on different shapes on a square grid. We completely analyze this game for 2-by-$n$ rectangles.

The \textit{Remove-an-Edge} game in Section~\ref{sec:re} is played on a graph. We analyze this game for complete, star, path, and cycle graphs. For a complete graph this game is a no-strategy game with a fixed number of moves. For a star graph this is a trivial game with one move. For a path graph this game is equivalent to the domino-covering game. The most interesting case we study is this game on the cycle graph.

The \textit{No-Factor} game in Section~\ref{sec:nf} is played on a range of numbers. The second player can be proven to have a winning strategy by a stealing strategy argument.

\section{Preliminaries}

We study impartial combinatorial games. There are two players and they have the same set of moves available to them.

A \emph{P-position} is a position from which the \emph{previous} player wins given perfect play. All terminal positions are P-positions. An \emph{N-position} is a position from which the \emph{next} player wins given perfect play. When we play we want to end our move with a P-position and want to see an N-position before our move. Any move from a P-position should be to an N-position. Also, there should exists a move from an N-position to a P-position.

Many games are solved by first providing a set of alleged P-positions and N-positions. To prove that this set is correct, it is enough to show that there exists a move from an N-position to a P-position and that any move from a P-position is to an N-position \cite{BCG, AND}. All the terminal positions are P-positions.

\section{Chocolate Stones game}\label{sec:cs}

Many combinatorial games use stones or tokens \cite{BCG, AND}. There is also a famous game of Chomp that is played on a chocolate bar. In honor of both types of games we call this game Chocolate Stones.

\textbf{Game.} Fix number $m$. Given a pile of $N$ chocolate stones a player is allowed to take away any number of stones from remainder of $N$ modulo $m$ up to $m$. If $N$ is divisible by $m$, a player is only allowed to take $m$ stones. The game ends when there are no moves.

\textbf{Example.} If $m=2$, then the players are allowed to take 1 or 2 stones from an odd pile and 2 stones from an even pile.

We can define a chocolate value $f$ of a pile of $N$ stones to be the ceiling of $\frac{N}{m}$: $f(N) = \lceil \frac{N}{m}\rceil$.  If  $N \neq 0$, then $f \neq 0$. That means the game continues while $f > 0$.

\begin{theorem}
The Chocolate Stones game is a no-strategy game. The P-positions are such that $f(N)$ is even.
\end{theorem}

\begin{proof}
We can show that $f$ decreases by 1 for every move.

A player has a pile $N=km + a$, where $a$ is the remainder of $N$ modulo $m$, if $N$ is not divisible by $m$. If $N$ is divisible by $m$, we set $a=m$. Therefore, $N$ has a chocolate value of $k + 1$.  The player can take at minimum $a$ and at maximum $m$ stones.  The result is a number in the range $[(k - 1)m + a,km]$. The chocolate value of every number in this range is $k$. Therefore, one move is equivalent to subtracting 1 from the chocolate value. Thus, the winner is determined from the start: if the chocolate value is odd, the first player wins, otherwise the second player wins.
\end{proof}

\section{Diamond game}\label{sec:diamonds}

In this game tokens are placed on a plane. They are not treated as a pile like in many games. The geometry of the configuration plays a role.

\textbf{Game}. Tokens are put on the intersections of the grid lines on a square grid. A player is allowed to take any non-empty row or column of tokens. The game ends when there are no moves.

We call it the Diamond game, because the main shape we study is a diamond when the tokens are put on every point with integer coordinates $a$ and $b$, such that $|a| + |b| \leq c$, see Figure~\ref{fig:5d} for an example with $c = 2$. We also study generalizations of this game, where the starting position can form other shapes.

\begin{figure}[h!]
\centering
\includegraphics[scale=0.4]{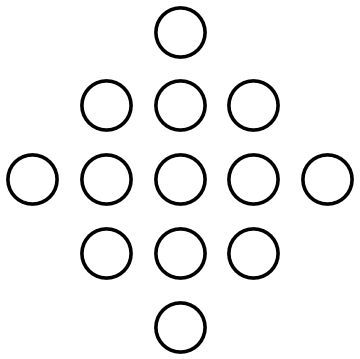}
\caption{5-diamond}
\label{fig:5d}
\end{figure}

The following lemma covers a special symmetric shape.

\begin{lemma}\label{lemma:noaxessym}
If the shape does not contain tokens on the axes and is symmetric with respect to both axes, then the second player wins.
\end{lemma}

\begin{proof}
If the first player makes a horizontal move, the second player mirrors the move with respect to $x$ axis. If the first player makes a vertical move, the second player mirrors the move with respect to $y$ axis. In any case after two moves the shape is symmetric with respect to both axes again, and can not be removed in one move.
\end{proof}

Let us go back to the diamond shapes. We can enumerate diamonds by the longest row. For example, we call the diamond in Figure~\ref{fig:5d} the 5-diamond. For the 1-diamond, the second player loses. We show later that for larger diamonds the second player wins. We call the token in position $(0,0)$, \textit{the center token}.

Here is the strategy for the second player.

\textbf{Strategy.} 
\begin{enumerate}
\item \textbf{End Game.} If tokens form a line, take the line. If tokens form two lines in one axial direction, then take a line in another direction.
\item \textbf{Mid-Game.} If the first player takes one of the axis, the second player takes the other. If the first player takes a line that is not the axis, but parallel to axis $a$, then the second player takes the mirror image of the line with respect to $a$.
\end{enumerate}

\begin{lemma}
The strategy above is a winning strategy for the second player.
\end{lemma}

\begin{proof}
At the beginning the shape is symmetric with respect to both axis. If during the Mid-Game the first player removes one axis line, then the second axis line is not empty. The second player removes the second axis line. As a result, after these two moves the shape is symmetric with respect to both axes, and by Lemma~\ref{lemma:noaxessym} the second player wins. 

Now we assume that the first player does not remove an axis in the Mid-Game. At some point after the first player's move, the result will be one axis line and another parallel line. Without loss of generality we can assume that these two lines are horizontal. Moreover, the figure is symmetric with respect to the vertical axis. The figure also contains the center. The horizontal axis line contains an odd number of tokens. In this situation, no player wants to remove a horizontal line as it would mean a loss. So the players keep removing vertical lines making sure that at least one vertical line with two tokens stays. (One such line exists: it is a vertical axis line.) As there is an odd number of vertical lines available, the second player is guaranteed to be last.
\end{proof}

We can also solve this game for other starting configurations. First we consider a cross, where tokens are placed on both axes, including the center. Without loss of generality we assume that tokens form one consecutive segment on each axis. The game depends only on two numbers $m$ and $n$: the number of tokens on each axis.

\begin{theorem}
If the starting position is a cross, then the P-positions are such that $m+n$ is even and both $m$ and $n$ are more than 1.
\end{theorem} 

\begin{proof}
No player wants to remove one of the axes, while other tokens are present. That means they remove one token on each move in such a way that the rest of the tokens do not form a line. This continues until the players get to the cross with $m=n=2$. At this point the second player wins.
\end{proof}

We now consider the game on an $m$-by-$n$ rectangle.

\begin{theorem}
If the starting position is an $m$-by-$n$ rectangle, then the P-positions are such that $m+n$ is even and both $m$ and $n$ are more than 1.
\end{theorem}

\begin{proof}
A rectangle 1-by-$n$ is an N-position and a 2-by-2 square is a P-position. Consider a 2-by-$m$ rectangle. In one move the players either get to a 1-by-$m$ or 2-by-$(m-1)$ rectangle. It follows that rectangles 2-by-$2k$ are P-positions. If both $m$ and $n$ are more than 2, then each move reduces the sum $m+n$ by 1. The theorem follows.
\end{proof}

As one may notice the P-positions for a cross and a rectangle are defined by the same rules.

\section{Demon Money game}\label{sec:dm}

Two people live in a city ruled by evil demons. Every day the demons require them to pay the square root of their total money in taxes. The two people pool their money and alternate paying the taxes every day. However, because the money are in gold coins, the person paying the taxes may choose to round the amount paid up or down. The person who pays the demons zero coins gets executed by the tyrannous demon king.

\textbf{Game.} Given a pile of $N$ coins players are allowed to take $\sqrt{N}$ coins rounded up or down. If $N = k^2$, then the only move is to take $k$ coins.

\textbf{Examples.} Empty pile is a P-position. It follows that piles of 1 or 2 coins are N-positions. From here, 3 and 4 are P-positions. Then 5, 6, and 7 are N-positions. Then 8, 9, 10 are P-positions and so on.

\begin{theorem}
For integer $k > 0$, the numbers $[k^2+k-1,(k+1)^2-2]$ are N-positions. The numbers $[k^2-1, k^2+k -2]$ are P-positions.
\end{theorem}

\begin{proof}
The terminal position is 0, which corresponds to a P-position for $k=0$.

Consider N-positions for given $k$: they are numbers between $k^2$ and $(k+1)^2$. Therefore, one move subtracts $k$ or $k+1$.

Subtracting $k$ from every element in the segment $[k^2+k-1,(k+1)^2-2]$ we get the segment $[k^2-1,k^2 +k-1]$. All the numbers but the last one are P-positions. That means we get to a P-position by subtracting $k$ from any number except the biggest number in the segment. In the latter case we subtract $k+1$, and get to a P-position. Thus, from any N-position there exists a move to a P-position.

Now we need to show that any move from a P-position goes to an N-position. Consider a P-position in the segment 
$[(k+1)^2-1, (k+1)^2+k -1]$. The smallest number we can reach in one move from this segment is $k^2 + k -1$ by subtracting $k+1$ from $(k+1)^2-1$. The largest number we can reach in one move from this segment is $(k+1)^2 -2$ by subtracting $k+1$ from $(k+1)^2+k -1$. In any case we end up in the segment $[k^2+k-1,(k+1)^2-2]$ which consists of N-positions. Thus any move from a P-position leads to an N-position.
\end{proof}

If the game is played optimally then the number of moves is fixed. If the starting position is in the range $[k^2+k-1,(k+1)^2-2]$, the game ends in $2k-1$ moves. If the starting position is in the range $[k^2-1, k^2+k -2]$, the game ends in $2k-2$ moves.

\section{Sum-from-Product}\label{sec:sfp}

This game combines the sum and the product of two numbers.

\textbf{Game.} Given a positive integer $n$, a player can choose any two integers $a$ and $b$, such that $ab=n$. The player subtracts $a+b$ from $n$, given that the result is positive. That is, the next player starts with a new number $n-a-b$. A player without a move loses.

We see that there are no moves available from prime numbers, 1, and 4. Otherwise, there exists a move. 

We wrote a program to calculate positions. P-Positions form a sequence: 1, 2, 3, 4, 5, 7, 11, 13, 16, 17, 19, 22, 23, 25, 27, 29, $\ldots$, which is now sequence A285304 in the OEIS \cite{OEIS}.

N-Positions form a sequence: 6, 8, 9, 10, 12, 14, 15, 18, 20, 21, 24, 26, 28, 30, $\ldots$, which is now sequence A285847 in the OEIS \cite{OEIS}.

\section{Remove-a-Square game}\label{sec:rs}

The following game is played on a square grid.

\textbf{Game.} Start with a shape made out of squares on a square grid. In a move a player is allowed to remove a full square. The player without a move loses.

The removed square can be of any size $k$-by-$k$.

\textbf{Example.} Suppose a starting position is a 3-by-3 square. The first player can win by removing the whole 3-by-3 square. The first player can prolong the pleasure by removing the 1-by-1 square in the center. After that all the available squares are 1-by-1, and the game ends in 9 moves. However, if the first player removes a 2-by-2 square, then there are no more full 2-by-2 squares available and the total number of moves is 6. If the first player removes a 1-by-1 square not in the center, then the second player can find a full 2-by-2 square to remove to guarantee a win. Thus the first player wins with the right play.

We study this game for a particular case of 2-by-$n$ rectangles. That means in one move a player can remove either a 1-by-1 or a 2-by-2 square.

If no one removes a 2-by-2 square, then the second player wins, as the total number of 1-by-1 squares is even. To control the win, the players must control the parity of the number of 2-by-2 squares removed.

\textbf{Example.} Suppose the players start with a 2-by-3 rectangle. If the first player removes a 1-by-1 square, the second player can remove another 1-by-1 square in the middle column. After that there are no full 2-by-2 squares, and the game ends in 6 moves. On the other hand, in the first move, the first player can remove a 2-by-2 square. After that there are two 1-by-1 squares left, and no full 2-by-2 squares. That means the game ends in three moves and the first player is guaranteed to win. 

\textbf{Example.} If we start with a 2-by-$2k$ rectangle, the first player wins. The first player can start by removing the middle 2-by-2 square. After that the game separates into two equal games, and the first player can match a move of the second player.

To solve this game we use the theory of Grundy numbers. Each game can be assigned a non-negative integer such that the P-positions are assigned zeros. The Grundy numbers are calculated recursively. The Grundy values of terminal positions are zero. Consider a position $A$. All moves from $A$ comprise a set of positions $\mathcal{A}$. Denote a set of Grundy numbers of $\mathcal{A}$ as $S$. Then the Grundy value of $A$ is $\mex (S)$, the least non-negative integer not in the set $S$. The sum of several games has a Grundy value that is equal to XOR (bitwise addition) of the Grundy values of those games. We use $\otimes$ to denote XOR.

Consider the first move. If a player removes a 2-by-2 square, then the result is the sum of two Remove-a-Square games on rectangles $2\times k$ and $2 \times (n-k-2)$. If a player removes a 1-by-1 square, then the result is the sum of three games on rectangles $2\times k$, $1\times 1$, and $2 \times (n-k-1)$. This allows us to calculate the Grundy values recursively noting that if the shape does not contain a 2-by-2 square then the Grundy number equals the parity of the total number of 1-by-1 squares.

We denote the Grundy number of a 2-by-$n$ rectangle as $G(n)$. Then $G(0) = G(1) = 0$ and
\[G(n) = \mex\{G(i) \otimes G(n-i-2), G(j) \otimes 1 \otimes G(n-j-1)\},\]
where $i \in [1,n-2]$ and $j \in [1,n-1]$.

We computed the first few Grundy values for this game represented in Table~\ref{table:rectGN}, where we have 12 columns to emphasize eventual periodicity. The value corresponding to $n = 12a+b$, where $1 \leq b \leq 12$ are in row $a+1$ and column $b$. The data shows Grundy numbers for $n$ from 1 to 192 inclusive. We can observe that for $83 \leq n  \leq 192$, the Grundy numbers are periodic: $G(n) = G(n-12)$.

\begin{table}
  \centering
    \begin{tabular}{| rrrrrrrrrrrr |}
    \hline
0&2&2&1&4&3&3&1&4&2&6&5 \\
0&2&7&1&4&3&3&1&4&7&7&5 \\
0&2&8&4&4&6&3&1&8&7&7&5 \\
0&2&2&1&4&6&3&1&8&2&7&5 \\
0&2&8&1&4&6&3&1&4&2&7&5 \\
0&2&8&1&4&6&3&1&8&7&7&5 \\
0&2&8&1&4&6&3&1&8&2&7&5 \\
0&2&8&1&4&6&3&1&8&2&7&5 \\
0&2&8&1&4&6&3&1&8&2&7&5 \\
0&2&8&1&4&6&3&1&8&2&7&5 \\
0&2&8&1&4&6&3&1&8&2&7&5 \\
0&2&8&1&4&6&3&1&8&2&7&5 \\
0&2&8&1&4&6&3&1&8&2&7&5 \\
0&2&8&1&4&6&3&1&8&2&7&5 \\
0&2&8&1&4&6&3&1&8&2&7&5 \\
0&2&8&1&4&6&3&1&8&2&7&5 \\
    \hline
    \end{tabular}
  \caption{Grundy numbers.}
\label{table:rectGN}
\end{table}

We prove that this pattern continues: the Grundy numbers are eventually periodic. 

\begin{theorem}
Given a 2-by-$n$ rectangle in the Remove-a-Square game, where $n  \geq 83$:
\[G(n) = G(n-12).\]
\end{theorem} 

\begin{proof}
Our data starts being periodic from 73 onward with period 12. We checked it until 199. We proceed by induction assuming that the statement is true for every number below $n > 199$.

Suppose $n > 172$. If we remove a 2-by-2 square, the result is the sum of two games on a $2\times k$ and $2 \times (n-k-2)$ rectangles for $0 \leq k \leq n-2$. The set of Grundy values for these games is $\{ G(k) \otimes G(n-k-2)\}$, for $0 \leq k \leq n-2$. One of the values $k$ or $n-k-2$ is more than 85. Let us assume, without loss of generality that $k > 85$. Then $G(k) \otimes G(n-k-2) = G(k-12) \otimes G(n-k-2)$. The game on two rectangles $2\times (k-12)$ and $2 \times (n-k-2)$ appears after removing a 2-by-2 square in the game on a 2-by-$(n-12)$ rectangle.

We can make a similar argument if our move is to remove a 1-by-1 square. The result is that the set of Grundy values after a move is the same for $n$ and for $n-12$. That means the $\mex$ function produces the same Grundy value for $n$ and $n-12$.
\end{proof}

\begin{corollary}
The P-positions correspond to numbers $n = 12a +1$.
\end{corollary}

The sequence of Grundy numbers is now sequence A286332 in the OEIS \cite{OEIS}.

\section{Remove-an-Edge}\label{sec:re}

This game is played on a simple graph.

\textbf{Game.} In one turn, a player is allowed to remove two neighboring vertices with all edges coming out of them. The player who does not have a turn loses.

\textbf{Example.} Consider a star graph $S_n$ with $n > 1$ vertices. In one move the first player has to remove the center of the star and all the edges. After the first move the graph has $n-2$ isolated vertices. The first player wins in one move.

\textbf{Example.} Consider a complete graph $K_n$ with $n$ vertices. A move turns it into a complete graph with $n-2$ vertices. This is a no-strategy game. The first player wins if $n = 4k+2$ or $4k+3$. Otherwise, the second player wins.

\textbf{Example.} Consider a path graph $P_n$ with $n$ vertices. This game is equivalent to a known game called \textit{domino covering} on a 1-by-$n$ rectangle. In the domino-covering game, the players take turns placing non-overlapping dominoes, aka 1-by-2 rectangles, on the board which is a 1-by-$n$ rectangle. The loser is the first player unable to move. The P-positions in this game are described in sequence A215721 \cite{OEIS}. For $n > 14$, $\text{A215721}(n) = \text{A215721}(n-5) + 34$. Equivalently, the sequence consists of numbers \{0,1,15,35\} and all the positive integers congruent to \{5,9,21,25,29\} modulo 34.

In particular, all P-positions are odd. If $n$ is even, the first player can remove the center edge dividing the game into two equivalent games. Therefore, the first player wins for even $n$.

The Grundy numbers for this game can be computed recursively. If we denote $G(n)$ the Grundy value for the 1-by-$n$ rectangle, then $G(0) = G(1) = 0$ and
\[G(n) = \mex \{G(i) \otimes 1 \otimes G(n-i-2)\}.\]
where $i \in [1,n-2]$.

The Grundy values of this game on a path are described in sequence A002187 \cite{OEIS}, that is $G(n) = \text{A002187}(n-1)$. The sequence is eventually periodic with period 34.

The object of our interest is this game on the cycle graph with $n$ vertices: $C_n$. The following lemma is straightforward:

\begin{lemma}
After the fist move, the Remove-an-Edge game on a cycle graph $C_n$ is equivalent to the Remove-an-Edge game on the path graph $P_{n-2}$.
\end{lemma}

For the Remove-an-Edge played on a cycle graph, the N-positions are $\text{A215721}(n+2)$. This sequence is sequence  A274161 in the OEIS \cite{OEIS}: 2, 3, 7, 11, 17, 23, 27, 31, 37, 41, 45, 57, $\ldots$.

In the OEIS the sequence A274161 is described as P-positions in a different game which is called \textit{the edge-delete game}. The edge-delete game is defined on a graph $G$ as follows: Two players alternate turns, permanently deleting one edge from $G$ on each turn. Unlike our game the vertices are not removed. The game ends when a vertex is isolated in what remains of $G$. The player whose deletion creates an isolated vertex loses the game. 

The sequence A274161 is the P-positions of the edge-delete game played on a path graph $P_n$. Let us consider a domino-covering game on the 1-by-$n+2$ rectangle. We can correspond a vertex to each square cell and connect the vertices if the cells share an edge. Then we match deleting an edge to covering the vertices connected by this edge by a domino. The rule that does not allow isolated vertices means that the dominoes cannot overlap and cannot cover the ending cells. That means the edge-delete game on the path $P_n$ is the same as the domino-covering game on the 1-by-$n$ rectangle, where dominoes are not allowed to cover the end squares. The latter game is the same as the domino-covering game on the 1-by-$(n-2)$ rectangle without any end restrictions.

The sequence of P-positions of the Remove-an-Edge game on $C_n$ is the complement of sequence A274161.

\section{No-Factor game}\label{sec:nf}

\textbf{Game.} Write out integers 1 through $n$. On a player's turn, s/he may remove any set of numbers so long as all numbers removed have no proper factors existing at the beginning of the player's turn. The person who does not have a move loses.

This game can be solved using a strategy stealing argument. Surprisingly, it is the second player who is stealing.

\begin{lemma}
For $n > 1$, the second player wins in the no-factor game. 
\end{lemma}

\begin{proof}
The first turn must be to take 1. If $n=1$, then the first player wins. Otherwise, suppose the first player has a winning strategy. Then let the second player take the prime number $p$ that is more than $n/2$. Such a prime number always exists for $n >1$ due to Bertrand postulate. After that there exists a move for the first player to win by our assumption. As $p$ does not have any common factors with the rest of the numbers, that means that on the second move the second player can take $p$ and whatever the first player takes on the third move, thus stealing a winning strategy.
\end{proof}

\section*{Acknowledgments}
This project was part of the PRIMES STEP program. We are thankful to the program for allowing us the opportunity to conduct this research.

\end{document}